\documentclass[12pt, reqno,twoside]{article}
\usepackage{amssymb}
\usepackage{amsmath}
\numberwithin{equation}{section}
\usepackage{amsthm}
\usepackage{graphicx}
\usepackage{ifthen}
\usepackage{cite}
\usepackage{tikz}
\usetikzlibrary{calc, positioning}
\usepackage{todonotes}
\usepackage{amsmath}
\usepackage{amsfonts}
\usepackage{amssymb}
\usepackage{amsmath,bbm,amssymb,amsxtra}
\usepackage{mathrsfs}
\usepackage{enumerate}
\usepackage{caption}
\allowdisplaybreaks
\usepackage{epsfig}
\usepackage{ae}
\usepackage{epstopdf}
\usepackage[left=2.5cm, right=2.5cm, top=2.5cm, bottom=2.5cm]{geometry}

\def\vint{\mathop{\mathchoice%
         {\setbox0\hbox{$\displaystyle\intop$}\kern 0.22\wd0%
          \vcenter{\hrule width 0.6\wd0}\kern -0.82\wd0}%
         {\setbox0\hbox{$\textstyle\intop$}\kern 0.2\wd0%
          \vcenter{\hrule width 0.6\wd0}\kern -0.8\wd0}%
         {\setbox0\hbox{$\scriptstyle\intop$}\kern 0.2\wd0%
          \vcenter{\hrule width 0.6\wd0}\kern -0.8\wd0}%
         {\setbox0\hbox{$\scriptscriptstyle\intop$}\kern 0.2\wd0%
          \vcenter{\hrule width 0.6\wd0}\kern -0.8\wd0}}%
         \mathopen{}\int}

\newcommand{\diam}{\text{\rm diam\,}}

\newtheorem{thm}{Theorem}[section]

\newtheorem{claim}{Claim}[section]
\newtheorem{rem}[thm]{Remark}
\newtheorem{rems}{Remarks}[section]
\newtheorem{cor}[thm]{Corollary}
\newtheorem{lem}[thm]{Lemma}
\newtheorem{prop}[thm]{Proposition}
\newtheorem{ques}[thm]{Quention}

\newtheorem{subclaim}{Subclaim}
\newtheorem{conj}[equation]{Conjecture}
\newtheorem{case}{Case}[section]
\newtheorem*{mysolution}{Solution}
\newtheorem{step}{Step}[section]
\theoremstyle{definition}
\newtheorem{defn}[thm]{Definition}

\newcounter {own}
\def\theown {\thesection       .\arabic{own}}

\newenvironment{pf}[1][]{%
 \vskip 3mm
 \noindent
 \ifthenelse{\equal{#1}{}}%
  {{\slshape Proof. }}%
  {{\slshape #1.} }%
 }%
{\qed\bigskip}

\newcounter{alphabet}

\makeatletter

\newcommand{\RNum}[1]{\uppercase\expandafter{\romannumeral #1\relax}}

\def\be{\begin{equation}}
\def\ee{\end{equation}}

\newcommand{\ben}{\begin{enumerate}}
\newcommand{\een}{\end{enumerate}}

\newcommand{\blem}{\begin{lem}}
\newcommand{\elem}{\end{lem}}
\newcommand{\bthm}{\begin{thm}}
\newcommand{\ethm}{\end{thm}}
\newcommand{\bcor}{\begin{cor}}
\newcommand{\ecor}{\end{cor}}
\newcommand{\beg}{\begin{examp}}
\newcommand{\eeg}{\end{examp}}
\newcommand{\begs}{\begin{examples}}
\newcommand{\eegs}{\end{examples}}
\newcommand{\bdefe}{\begin{defn}}
\newcommand{\edefe}{\end{defn}}
\newcommand{\bprob}{\begin{prob}}
\newcommand{\eprob}{\end{prob}}
\newcommand{\bques}{\begin{ques}}
\newcommand{\eques}{\end{ques}}
\newcommand{\bei}{\begin{itemize}}
\newcommand{\eei}{\end{itemize}}
\newcommand{\bcl}{\begin{claim}}
\newcommand{\ecl}{\end{claim}}
\newcommand{\bscl}{\begin{subclaim}}
\newcommand{\escl}{\end{subclaim}}
\newcommand{\bca}{\begin{case}}
\newcommand{\eca}{\end{case}}
\newcommand{\bstep}{\begin{step}}
\newcommand{\estep}{\end{step}}
\newcommand{\bsol}{\begin{mysolution}}
\newcommand{\esol}{\end{mysolution}}
\newcommand{\bcon}{\begin{conj}}
\newcommand{\econ}{\end{conj}}
\newcommand{\bcons}{\begin{conjs}}
\newcommand{\econs}{\end{conjs}}
\newcommand{\bprop}{\begin{prop}}
\newcommand{\eprop}{\end{prop}}
\newcommand{\br}{\begin{rem}}
\newcommand{\er}{\end{rem}}
\newcommand{\brs}{\begin{rems}}
\newcommand{\ers}{\end{rems}}
\newcommand{\bo}{\begin{obser}}
\newcommand{\eo}{\end{obser}}
\newcommand{\bos}{\begin{obsers}}
\newcommand{\eos}{\end{obsers}}

\newcommand{\bpf}{\begin{pf}}
\newcommand{\epf}{\end{pf}}

\newcommand{\ba}{\begin{array}}
\newcommand{\ea}{\end{array}}
\newcommand{\beq}{\begin{eqnarray}}
\newcommand{\beqq}{\begin{eqnarray*}}
\newcommand{\eeq}{\end{eqnarray}}
\newcommand{\eeqq}{\end{eqnarray*}}

\newcounter{minutes}
\divide\time by 60
\newcounter{hours}
\multiply\time by 60 \addtocounter{minutes}{-\time}

\begin{document}
	
\bibliographystyle{amsplain}


	
	\title{\Large\bf  Equivalent characterizations of John and uniform domains in doubling metric spaces
		
		\footnotetext{\hspace{-0.35cm}
			$2020$ {\it Mathematics Subject Classification}:  30C45; 30C20.  
			\endgraf{{\it Key words and phrases}: uniform domain, John domain, doubling metric space, quasihyperbolic metric, locally quasiconvex domain. }
			
		}
	}
	\author{Yaxiang Li, Yahui Sheng, Zhuang Wang}
	\date{ }
	\maketitle
\begin{abstract}
	 In this paper, we characterize John and uniform domains in doubling metric spaces. Specifically, we show that a locally quasiconvex domain in a doubling metric space is length John if and only if it is diameter John. For uniform domains, we  prove that a domain in a doubling metric space is length uniform if and only if it is diameter uniform (or distance uniform) and locally quasiconvex. Moreover, in a doubling length metric space, we refine this result by showing that a domain is length uniform (resp. John) if and only if it is diameter uniform (resp. John).
\end{abstract}


\section{Introduction}

John domains in $\mathbb{R}^n$ were originally introduced in 1961 by John~\cite{Jo61} in connection with his work on elasticity, and the term is due to Martio and Sarvas~\cite{MS78}, who at the same time introduced the closely related class of uniform domains. Today, John and uniform domains are fundamental objects in geometric function theory and analysis on metric spaces. They are defined by purely metric conditions that capture the borderline regularity at which classical analytic tools become fully available in non‑smooth settings. In particular, they underpin Sobolev embeddings and Poincaré inequalities~\cite{BK95,HK98}, Jones's extension theorem for $\mathrm{BMO}$ functions and Sobolev functions~\cite{Jo80,Jo81}, the boundary behavior of quasiconformal and quasihyperbolic mappings~\cite{GH87,GO79}, and the uniformization of Gromov hyperbolic spaces into uniform domains in $\mathbb{R}^n$, due to Bonk–Heinonen–Koskela \cite{BHK01}. Over the last three decades, the ideas surrounding John and uniform domains have been extended well beyond the Euclidean setting. Two major developments mark this extension. The first decisive step was V\"ais\"al\"a's systematic study of Banach space domains~\cite{V91,V99}. The second, and more recent, development is the program of analysis on metric spaces developed by Heinonen, Koskela, Shanmugalingam, Tyson and others~\cite{H01,HKST15}, whose goal is to identify the weakest geometric framework that still supports a meaningful first‑order calculus.

John and uniform domains can be defined in three natural ways: by controlling the length, the diameter, or the chord‑distance of curves joining pairs of points. We refer to these as the length, diameter and distance versions, respectively. The precise definitions are collected in Section~\ref{pre}. 
The relations among these formulations are well understood in $\mathbb{R}^n$. For uniform domains all three are equivalent (see~\cite{Ma80,MS78}). For John domains, length and diameter versions are quantitatively equivalent, but a distance John domain need not be diameter John (see \cite[Theorem~2.14 and Example~2.15]{NV91}). When one passes to infinite-dimensional Banach spaces, further differences emerge: V\"ais\"al\"a constructed a diameter uniform domain that is neither length John nor quasiconvex \cite[Properties~3.13, 3.14 and 3.18]{Va04}. That means, in general Banach spaces, a diameter John domain need not be length John, and a diameter uniform domain need not be length uniform. Subsequent work by Zhou and Rasila \cite[Theorem 1.2]{ZR20} established that a domain in a Banach space is length uniform if and only if it is both diameter uniform and natural (see \cite[p.~396]{V92} for the definition of natural domain). 

The aforementioned counterexamples in Banach spaces reveal a fundamental issue: without sufficient geometric control, the diameter or distance conditions are too weak to constrain the length of curves. Metric spaces are essential in modern geometric analysis because they allow us to work with spaces that may lack a linear structure. Therefore, when we move from Banach spaces to general metric spaces, the focus shifts: we must identify which geometric “thickness” conditions can restore the equivalence between the different formulations of John and uniform domains. In recent years, significant progress has been made in understanding the geometric behavior of metric spaces purely in terms of their intrinsic geometric properties. For instance,  Gao–Guo–Huang–Wang \cite{GGHW26} and Guo–Huang–Wang \cite{GH25III} have deeply studied quasihyperbolic geometry and Gromov hyperbolicity in metric spaces, and in particular obtained characterizations of quasihyperbolic John and uniform domains. For more related studies on John spaces, quasihyperbolic geodesics and Gromov hyperbolicity,  see also \cite{ZP24,ZPL25,ZLR22}. These works show that,  under suitable metric constraints (such as the doubling property), the quasihyperbolic metric can effectively control the geometric behavior of domains. These studies suggest that imposing local connectivity and volume control may compensate for the lack of linear structure. Consequently, the different versions of John and uniform domains may become equivalent in such settings. This naturally raises the following questions:

\begin{quote}
	\itshape
	On a metric space, which condition ensures that the length and diameter formulations of John domains are equivalent? Furthermore, for uniform domains, which condition guarantees the equivalence of all three formulations?
\end{quote} 

We consider these problems in the setting of doubling metric spaces.  Our first main result concerns John domains. For the definitions of doubling metric spaces and locally quasiconvex domains, we refer to Definition \ref{def:doubling} and Definition \ref{def:loc-qc}, respectively.

\begin{thm}\label{diam John is John}
	Let $Q>1$ and let $X$ be a $Q$-doubling metric space. If $G\subsetneq X$ is a locally $(q,\lambda)$-quasiconvex domain, then $G$ is a length John domain if and only if it is a diameter John domain.
\end{thm}

Under the assumptions of Theorem~\ref{diam John is John}, the distance John condition is not equivalent to the diameter John condition, since even in $\mathbb{R}^n$ a distance John domain need not be a diameter John domain; see \cite[Example~2.15]{NV91}.

The second main result is an analogue for uniform domains, in which all three definitions coincide.

\begin{thm}\label{distance uniform is uniform}
	Let $Q>1$, let $X$ be a $Q$-doubling metric space, and let
	$G\subsetneq X$ be a domain. Then the following statements are
	equivalent.
	\begin{enumerate}
		\item[\textup{(i)}]   $G$ is length uniform;
		\item[\textup{(ii)}]  $G$ is diameter uniform and locally
		quasiconvex;
		\item[\textup{(iii)}] $G$ is distance uniform and locally
		quasiconvex.
	\end{enumerate}
\end{thm}
The above two theorems yield a clean corollary for the important subclass of doubling length spaces, in which the local quasiconvexity condition is automatically satisfied (see Lemma~\ref{qc is locally qc}). For the definition of length spaces, see Section~\ref{pre}.

\begin{cor}\label{John-cor}
	Let $Q>1$ and let $X$ be a $Q$-doubling  length metric space. Suppose that $G\subsetneq X$ is a domain. Then the following hold.
	\begin{itemize}
		\item [\textup{(1)}] $G$ is a length John domain if and only if it is a diameter
		John domain.
		\item[\textup{(2)}]  $G \text{ is  length uniform }
		\iff
		G \text{ is  diameter uniform }
		\iff
		G \text{ is  distance uniform}.$
	\end{itemize}
\end{cor}

\begin{rem}\rm
\begin{enumerate}
\item[\textup{(a)}]  In Theorem \ref{diam John is John} and Theorem \ref{distance uniform is uniform}, the $Q$-doubling condition  cannot be removed, as illustrated by an example of V\"ais\"al\"a in \cite[Properties~3.2, 3.13, 3.14 and 3.18]{Va04}, in which a locally quasiconvex, diameter uniform domain is constructed that is neither length John nor length uniform. 
Moreover, the equivalence between diameter uniform and distance uniform domains, which underlies the implication (ii) $\iff$ (iii) in Theorem~\ref{distance uniform is uniform}, in fact holds in any metric space without the doubling hypothesis or local quasiconvexity (see Lemma \ref{distance uniform is diameter uniform}).
	
  \item[\textup{(b)}]The proofs of  Theorems \ref{diam John is John} and  \ref{distance uniform is uniform} rely on a chain argument (Lemma \ref{ball sequence}) and a covering argument (Lemma \ref{doubling property}) to replace diameter John (or uniform) curves by length John (or uniform) curves. Lemma \ref{ball sequence} is of independent interest:  it provides a decomposition of a curve into a chain of overlapping balls such that non‑adjacent balls are pairwise disjoint. Such a decomposition provides a useful tool for converting diameter bounds into length bounds in doubling metric spaces.
\end{enumerate}
\end{rem}

This paper is organized as follows. Section 2 introduces the basic concepts, terminologies, and preliminary results that will be used later. Section 3 is devoted to the proofs of Theorems \ref{diam John is John} and \ref{distance uniform is uniform}, as well as the proof of Corollary \ref{John-cor}.

\section{Preliminaries}\label{pre}

Throughout this paper, $(X,d)$ denotes a metric space;
 when no
confusion arises, we  simply write $X$ for $(X,d)$. The distance between $x,y\in X$ is denoted by $d(x,y)$. For
$x\in X$ and $r>0$, we denote by $B(x,r)$, $\overline{B}(x,r)$ and $S(x,r)$ the sets $\{y\in X:  d(x, y)<r\}$,  $\{y\in X:  d(x, y)\leq r\}$ and $\{y\in X:  d(x, y)=r\}$, respectively. 

By a \emph{curve} in $X$ we always
mean a continuous map $\gamma:I=[a,b]\to X$; we identify $\gamma$
with its image $\gamma(I)$ when no ambiguity arises,  and the length is
\[
    \ell(\gamma)=\sup\sum_{i=1}^{n} d\left( \gamma(t_{i}),\gamma(t_{i-1})\right) ,
\]
where the supremum is taken over all partitions
$a=t_{0}<t_{1}<\dots<t_{n}=b$. The curve is called
\emph{rectifiable} if $\ell(\gamma)<\infty$, and a metric space $X$
is said to be \emph{rectifiably connected} if every two points of
$X$ can be joined by a rectifiable curve.  A domain $G\subset X$ is called rectifiably connected if every two points of $G$ can be joined by a rectifiable curve in $G$.

Let $\gamma$ be a rectifiable curve. The length function
$s_{\gamma}:[a,b]\to[0,\ell(\gamma)]$ is defined by
$s_{\gamma}(t)=\ell(\gamma[a,t])$, and there is a unique curve
$\gamma_{s}:[0,\ell(\gamma)]\to X$ such that
$\gamma=\gamma_{s}\circ s_{\gamma}$. The curve $\gamma_{s}$ is
called the \emph{arc-length parametrization} of $\gamma$; by
construction $\ell\bigl(\gamma_{s}[0,t]\bigr)=t$ for every
$t\in[0,\ell(\gamma)]$. The line integral over $\gamma$ of a Borel
function $\varrho:X\to[0,\infty)$ is then
\[
    \int_{\gamma}\varrho ds=\int_{0}^{\ell(\gamma)}\varrho\circ\gamma_{s}(t)dt.
\]

The quasihyperbolic metric is a fundamental tool in geometric
function theory and in analysis on metric spaces. It was introduced
by Gehring and Palka~\cite{GP76} for domains in
$\mathbb{R}^{n}$ and has since been extensively studied in Banach
spaces (see~\cite{V90,V91,V92,V98,V99}) and in metric spaces
(see~\cite{BH07,HL15,HL16,HRWZ18} and the references therein).

For a rectifiably connected domain $G\subsetneq X$, the
\emph{quasihyperbolic length} of a rectifiable curve $\gamma$ in $G$
is defined by
\[
    \ell_{\mathrm{qh}}(\gamma)=\int_{\gamma}\frac{ds}{\delta_{G}(z)},
\]
where $\delta_G(z)=\inf\{d(z, y): y\in X\setminus G\}$ denotes the distance from $z$ to $X\setminus G$. Observe that $\delta_{G}(x)>0$ and $B( x,\delta_{G}(x)) \subset G$ for
every $x\in G$.  The \emph{quasihyperbolic metric} between $x,y\in G$ is
\begin{equation}\label{eq:2.1}
    k_{G}(x,y)=\inf_{\gamma}\ell_{\mathrm{qh}}(\gamma),
\end{equation}
where the infimum is taken over all rectifiable curves $\gamma$ in
$G$ joining $x$ and $y$.

Next we collect the definitions of John domains and uniform domains. Throughout, $G\subsetneq X$
is a domain and $c\geq 1$ is a constant.

\begin{defn}\label{def:length-John}
A domain $G\subsetneq X$ is called a \emph{length $c$-John domain}
if every pair of points $x,y\in G$ can be joined by a rectifiable
curve $\gamma\subset G$ satisfying
\begin{equation*}
    \min\left\lbrace \ell(\gamma[x,z]),\ell(\gamma[z,y])\right\rbrace 
  \leq c\,\delta_{G}(z)
    \tag{$\mathrm{J}_{\ell}$}, \quad \mbox{for all}\,\, z\in\gamma,
\end{equation*}
where $\gamma[u,v]$ denotes the subcurve of $\gamma$ with endpoints $u$ and $v$. Such a curve $\gamma$ is called a \emph{length $c$-John curve}.
\end{defn}

\noindent
Replacing the length $\ell$ by $\diam$ on the left-hand side of
$(\mathrm{J}_{\ell})$ gives the notion of a \emph{diameter
$c$-John domain} (and \emph{diameter $c$-John curve}):
\begin{equation*}
    \min\left\lbrace \diam(\gamma[x,z]),\diam(\gamma[z,y])\right\rbrace 
    \leq
    c\,\delta_{G}(z)
    \tag{$\mathrm{J}_{\diam}$} \quad \mbox{for all}\,\, z\in\gamma,
\end{equation*}
while replacing the left-hand side by $\min\{d(x,z),d(z,y)\}$ gives
the notion of a \emph{distance $c$-John domain} (and \emph{distance
$c$-John curve}):
\begin{equation*}
    \min\left\lbrace d(x,z),d(z,y)\right\rbrace
    \leq
    c\,\delta_{G}(z)
    \tag{$\mathrm{J}_{d}$}\quad \mbox{for all}\,\, z\in\gamma.
\end{equation*}

\begin{defn}\label{def:length-uniform}
A domain $G\subsetneq X$ is called a \emph{length $c$-uniform
domain} if every pair of points $x,y\in G$ can be joined by a
rectifiable curve $\gamma\subset G$ that simultaneously satisfies
\begin{equation*}
    \ell(\gamma)\leq c\,d(x,y) \tag{$\mathrm{U}_{\ell}^{\mathrm{g}}$}
\end{equation*}
and
\begin{equation*}
    \min\left\lbrace \ell(\gamma[x,z]),\ell(\gamma[z,y])\right\rbrace 
    \leq
    c\,\delta_{G}(z)
    \tag{$\mathrm{U}_{\ell}^{\mathrm{loc}}$}\quad \mbox{for all}\,\, z\in\gamma.
\end{equation*}
Note that the condition $(\mathrm{U}_{\ell}^{\mathrm{loc}})$ is exactly the same as the condition $(\mathrm{J}_{\ell})$, the length John condition. A curve  satisfying $(\mathrm{U}_{\ell}^{\mathrm{g}})$ is called \emph{$c$-quasiconvex}.
A curve satisfying both $(\mathrm{U}_{\ell}^{\mathrm{g}})$ and $(\mathrm{U}_{\ell}^{\mathrm{loc}})$ is called a
\emph{length $c$-uniform curve}.
\end{defn}

\noindent
The notions of \emph{diameter $c$-uniform domain} and
\emph{distance $c$-uniform domain} are obtained from
Definition~\ref{def:length-uniform} by the following standard
substitutions.

\begin{itemize}\itemsep4pt
  \item \emph{Diameter uniform}: replace $\ell(\gamma)$ by
        $\diam(\gamma)$ in
        $(\mathrm{U}_{\ell}^{\mathrm{g}})$ and
        $\min\left\lbrace \ell(\gamma[x,z]),\ell(\gamma[z,y])\right\rbrace $ by
        $\min\left\lbrace \diam(\gamma[x,z]),\diam(\gamma[z,y])\right\rbrace $ in
        $(\mathrm{U}_{\ell}^{\mathrm{loc}})$. The corresponding
        curve is called a \emph{diameter $c$-uniform curve}.
  \item \emph{Distance uniform}: replace $\ell(\gamma)$ by
        $\diam(\gamma)$ in
        $(\mathrm{U}_{\ell}^{\mathrm{g}})$ and
        $\min\left\lbrace \ell(\gamma[x,z]),\ell(\gamma[z,y])\right\rbrace$ by
        $\min\left\lbrace d(x,z),d(z,y)\right\rbrace $ in
        $(\mathrm{U}_{\ell}^{\mathrm{loc}})$. The corresponding
        curve is called a \emph{distance $c$-uniform curve}.
\end{itemize}

\begin{defn}\label{def:doubling}
A metric space $X$ is called \emph{$Q$-doubling} if for each ball
${B}(x,r)$, every $r/2$-separated subset of $ {B}(x,r)$ has cardinality at
most $Q$; here, a set $S\subset X$ is called
\emph{$\sigma$-separated}, $\sigma>0$, if $d(u,v)\geq\sigma$ for any
two distinct $u,v\in S$.
\end{defn}

\noindent
A standard volume comparison shows that $\mathbb{R}^{n}$ equipped
with the Euclidean distance is $2^{n}$-doubling.

Every rectifiably connected metric space $(X,d)$ admits a natural
(or intrinsic) metric, called its \emph{length metric}:
\[
    \rho(x,y):=\inf\bigl\{\ell(\gamma):
        \gamma\text{ a rectifiable curve joining $x$ and $y$ in }X
    \bigr\}.
\]
The space $(X,d)$ is called a \emph{length space} if
$d(x,y)=\rho(x,y)$ for all $x,y\in X$.


\begin{defn}\label{def:qc}
Let $c\geq 1$. A metric space $X$ (resp.\ a domain $G\subset X$) is
called \emph{$c$-quasiconvex} if every two points $x,y\in X$
(resp.\ $x,y\in G$) can be joined by a $c$-quasiconvex curve $\gamma$ in $X$
(resp.\ in $G$). 
\end{defn}

\begin{defn}\label{def:loc-qc}
Let $0<q<1$ and $\lambda\geq 1$. A rectifiably connected domain
$G\subsetneq X$ is called \emph{locally $(q,\lambda)$-quasiconvex}
if for every $x\in G$ and every pair of points
$y,z\in B(x,q\,\delta_{G}(x))$ there exists a
$\lambda$-quasiconvex curve $\gamma$ in $G$ joining $y$ and $z$.
\end{defn}

Clearly, quasiconvexity implies local quasiconvexity; the converse,
however, fails, as the following example shows. The space
$\bigl(\mathbb{R}^{2}\setminus\mathbb{R}_{+},\,|\cdot|\bigr)$ is
locally $(q,1)$-quasiconvex for every $q\in(0,1)$ but is not
quasiconvex, where
$\mathbb{R}_{+}=\{(x_{1},x_{2})\in\mathbb{R}^{2}:x_{1}\geq 0,\,
x_{2}=0\}$ and $|\cdot|$ is the Euclidean metric. On the other
hand, $(\mathbb{R}^{2},|\cdot|)$ itself is $1$-quasiconvex, which shows that quasiconvexity of the ambient space does not, in general, descend to its proper subdomains. However, every proper subdomain of a quasiconvex metric space is locally quasiconvex, as the following lemma shows.

\begin{lem}\rm\label{qc is locally qc}
	Let $X$ be a $c$-quasiconvex metric space and $G\subsetneq X$ be a domain. Then $G$ is locally $(q,c)$-quasiconvex for any $q\in\left( 0,1/(1+c)\right) $. 
\end{lem}
\begin{proof}
	Let $q\in(0,1/(1+c))$ and $z\in G$. For any $x,y\in B(z,q\,\delta_G(z))$, by the quasiconvexity of $X$, there is a curve $\gamma$ in $X$ joining $x$ and $y$ such that 
	$$\ell(\gamma)\leq c\,d(x,y).$$
	Since $d(x, y)<2q\,\delta_G(z)$, we have 
	$$\ell(\gamma)\leq 2cq\,\delta_G(z).$$
	For any $u\in\gamma$, it is clear that 
	$$d(z,u)\leq q\,\delta_G(z)+\frac{1}{2}\ell(\gamma)\leq(1+c)q\,\delta_G(z)<\delta_G(z).$$
	This means that $\gamma\subset B(z,\delta_G(z))\subset G$, which completes the proof.
\end{proof}
Since every length space is  $c$-quasiconvex for all $c>1$,  it follows from Lemma~\ref{qc is locally qc} that every proper subdomain of a length space is locally quasiconvex.

The following lemma contains comparison results between the metrics $d$ and $k_G$ in  metric spaces.
\begin{lem}\rm\label{property of k}
	Suppose that \( X\) is a metric space and  \( G \subsetneq X \) is a rectifiably connected domain.
	\begin{enumerate}
		\item[$(1)$] For each \( x, y \in G \), if $\gamma$ is a rectifiable curve joining $x$ and $y$, then we have
		\begin{equation}\label{QH prop 1}
			\left| \log\frac{\delta_G(x)}{\delta_G(y)}\right|\leq\log\left(1+\frac{d(x,y)}{\min\{\delta_G(x),\delta_G(y)\}} \right)\leq\log\left(1+\frac{\ell(\gamma)}{\min\{\delta_G(x),\delta_G(y)\}} \right)\leq \ell_{\text{qh}}(\gamma)
		\end{equation}
		and
		\begin{equation}\label{QH prop 2}
			\log\left( 1+\frac{d(x,y)}{\min\{\delta_G(x),\delta_G(y)\}}\right)\leq k_G(x,y).
		\end{equation}
		
		\item[$(2)$]  If \( z \in G \), $0<t<\frac{1}{2}$ and \( x,y\in \overline{B}(z, t\,\delta_G(z))  \), then
		\begin{equation}\label{QH prop 3}
			k_G(x, y)\geq \frac{1}{1+2t}\frac{d(x,y)}{\delta_G(z)}.
		\end{equation}

		\item[$(3)$]  Let $0<t, q<1$ and $\lambda\geq 1$ be constants. If $G$ is a locally $(q,\lambda)$-quasiconvex domain, then for each $z\in G$ and every pair of points $x,y\in \overline{B}( z,\frac{tq}{4\lambda}\,\delta_G(z))$, we have
		\begin{equation}\label{QH prop 4}
			k_G(x, y) \leq \frac{\lambda}{1-t} \frac{d(x,y)}{\delta_G(z)}.
		\end{equation}
	\end{enumerate}
\end{lem}

\begin{proof}
	$(1)$ The inequalities \eqref{QH prop 1} and \eqref{QH prop 2} follow from \cite[p. 9]{GH25III}; a similar proof is given in \cite[Lemma 2.2]{V90}. \medskip
	
	$(2)$ Let  \( \gamma : [0, \ell(\gamma)] \to G \) be a rectifiable curve  joining \( x, y \in G \) which is parameterized by its arc-length. Such a curve exists because $G$ is rectifiably connected. Then it suffices to show that
	\begin{equation}\label{lem-2}
		\frac{1}{1+2t}\frac{d(x,y)}{\delta_G(z)} \leq \int_0^{\ell(\gamma)} \frac{dr}{\delta_G(\gamma(r))}.
	\end{equation}
	We divide the proof into two cases:
	
	\textbf{Case 1.} \( \gamma \subset \overline{B} (z, 2t\,\delta_G(z))  \). In this case, for all \( u \in \gamma \),
	\[
	\delta_G(u) \leq d(z,u) + \delta_G(z) \leq (1+2t)\,\delta_G(z).
	\]
	Therefore, it follows that
	\[
	\int_0^{\ell(\gamma)} \frac{dr}{\delta_G(\gamma(r))} \geq \frac{\ell(\gamma)}{(1+2t)\,\delta_G(z)} \geq \frac{1}{1+2t}\frac{d(x,y)}{\delta_G(z)}.
	\]
	
	\textbf{Case 2.} \( \gamma \not\subset \overline{B}(z, 2t\,\delta_G(z)) \). From the connectedness of \( \gamma \), there exist  subcurves of $\gamma$ contained in $\overline{B}(z, 2t\,\delta_G(z))$  joining the spheres \( S(z, t\,\delta_G(z)) \) and \( S(z, 2t\,\delta_G(z)) \). We may assume that among these subcurves there exist two disjoint ones, denoted $\gamma_1$ and $\gamma_2$. Otherwise,  if no two such subcurves are disjoint (i.e., every pair intersects),  one can find  a curve $\gamma^\prime\subset \gamma$ joining $x, y$ such that $\gamma^\prime\subset \overline{B}(z, 2t\,\delta_G(z))$. Applying Case $1$ to $\gamma^\prime$, we obtain
	$$\int_0^{\ell(\gamma)} \frac{dr}{\delta_G(\gamma(r))} \geq\int_0^{\ell(\gamma^\prime)} \frac{dr}{\delta_G(\gamma^\prime(r))} \geq \frac{1}{1+2t}\frac{d(x,y)}{\delta_G(z)},$$
	which gives \eqref{lem-2}.
	
	For the two disjoint subcurves $\gamma_1$ and $\gamma_2$, it is clear that
	\[
	\ell(\gamma_1) \geq t\,\delta_G(z)  \quad \text{and}  \quad \ell(\gamma_2) \geq t\,\delta_G(z).
	\]
	Since \( d(x,y)\leq 2t\,\delta_G(z)\), it follows that
	\[
	\ell(\gamma_1)+\ell(\gamma_2) \geq 2t\,\delta_G(z)\geq  d(x,y).
	\]
	Moreover, because $\gamma_1$ and $\gamma_2$ belong to $\overline{B}(z, 2t\,\delta_G(z))$, 	for every \( u \in \gamma_1\cup\gamma_2 \) we have
	\[ \delta_G(u)\leq d(u, z)+\delta_G(z) \leq (1+2t)\,\delta_G(z). \]
	Consequently,
	$$ \int_0^{\ell(\gamma)} \frac{dr}{\delta_G(\gamma(r))}\geq \frac{\ell(\gamma_1)+\ell(\gamma_2)}{(1+2t)\,\delta_G(z)} \geq \frac{1}{1+2t}\frac{d(x,y)}{\delta_G(z)}.$$
	
	Combining these two cases, we obtain that \eqref{lem-2} holds, and hence the proof of $(2)$ is complete.\medskip

	$(3)$ Let $z\in G$  and let $x,y\in \overline{B}( z,\frac{tq}{4\lambda}\,\delta_G(z))$. Since $G$ is locally $(q,\lambda)$-quasiconvex, there exists a rectifiable curve \( \gamma \) in \( G \) joining \( x \) to \( y \) with \( \ell(\gamma) \leq \lambda \,d(x,y). \)
	
	For every \( u \in \gamma \), it is clear that
	\[
	d(u,z)\leq d(x,z)+\ell(\gamma)
	\leq d(x,z)+\lambda \,d(x,y) \leq (2\lambda+1)\frac{tq}{4\lambda}\,\delta_G(z)< t\,\delta_G(z),
	\]
	which implies that
	\[
	\gamma \subset B(z, t\,\delta_G(z)).
	\]
	Furthermore, for every \( u \in \gamma \), we have
	\[
	\delta_G(u) \geq \delta_G(z)-d(z,u) \geq (1-t)\,\delta_G(z).
	\]		
	Let \( \gamma_s : [0, \ell(\gamma)] \to G \) be the arc-length parametrization of \( \gamma \). Then
	\[
	k_G(x, y) \leq \int_0^{\ell(\gamma)} \frac{dr}{\delta_G(\gamma_s(r))} \leq \frac{\ell(\gamma)}{(1-t)\,\delta_G(z)} \leq  \frac{\lambda}{1-t}\frac{d(x,y)}{\delta_G(z)}.
	\]
	This completes the proof of $(3)$.
\end{proof}

Next, we establish that $\ell_k(\gamma)=\ell_{\text{qh}}(\gamma)$ for every rectifiable curve $\gamma$ in $G$, where $\ell_k(\gamma)$ denotes the length of $\gamma$ in the metric space $(G, k_G)$. This equality then  implies that $(G, k_G)$ is a length space.

\begin{prop}\rm\label{t1}
Let \( X \) be a metric space, and \( G \subsetneq X \) be a rectifiably connected domain. If $\gamma$ is a rectifiable curve in \( G \),
then $\ell_{k}(\gamma)=\ell_{\text{qh}}(\gamma)$; in particular,  \( (G, k_G) \) is a length space.
\end{prop}
\begin{proof}
	Let \( \gamma_s \) be the arc-length parametrization of \( \gamma \) and \( L=\ell(\gamma) \). Let \( 0 = t_0 < t_1 < \cdots < t_m = L \) be a partition of \([0, L]\). Then
	\[
	\sum_{j=1}^m k_G(\gamma_s(t_{j-1}), \gamma_s(t_j)) \leq \sum_{j=1}^m \ell_{\text{qh}}(\gamma_s[t_{j-1}, t_j]) = \ell_{\text{qh}}(\gamma).
	\]
	Hence \( \ell_{k}(\gamma) \leq \ell_{\text{qh}}(\gamma) \).
	
	Now we  turn to prove \( \ell_{k}(\gamma) \geq \ell_{\text{qh}}(\gamma) \). First, we give the following claim.
\begin{claim}\label{section3-claim}\rm
For any $0<\epsilon<\frac{1}{2}$, there exists a partition \( P: 0 = t_0 < t_1 < \cdots < t_n = L \) of \([0, L]\) with the following property: if we set \( x_i = \gamma_s(t_i), \gamma_i = \gamma_s[t_{i-1}, t_i] \) for \( 1 \leq i \leq n \), then 

\begin{equation}\label{first condition}
\ell_{\text{qh}}(\gamma) \leq \sum_{i=1}^n \frac{t_i-t_{i-1}}{\delta_G(\gamma_s(t_i))} + \epsilon
\end{equation}
and 
\begin{equation}\label{second condition}
\gamma_i \subset B(x_i, \epsilon\, \delta_G(x_i)).
\end{equation}

\end{claim}

For the proof of this claim, since both \(\gamma_s: [0, L]\to G\) and \(\delta_G: G\to (0, +\infty)\) are continuous, the composition \(\delta_G \circ \gamma_s:[0, L]\to (0, +\infty)\) is continuous on the compact interval \([0, L]\).  
		Consequently, this continuous function attains a positive minimum:  
		\[
		M_1 = \min_{t \in [0, L]} \delta_G(\gamma_s(t)).
		\]
Moreover, $1/(\delta_G\circ\gamma_s)$ is integrable on $[0, L]$.
		
		By the definition of the Riemann integral, there exists \(M_2 > 0\) such that for any partition \(P: 0 = t_0 < t_1 < \cdots < t_n = L\) with mesh size \(\|P\| := \max\limits_i (t_i - t_{i-1}) < M_2\), it holds that 
		\[
		\left| \sum_{i=1}^n \frac{t_i-t_{i-1}}{\delta_G(\gamma_s(t_i))} - \int_0^L \frac{1}{\delta_G(\gamma_s(t))} \, dt \right| < \frac{\epsilon}{2}.
		\]
	Therefore, for such partitions,
		\[
		\ell_{\mathrm{qh}}(\gamma) < \sum_{i=1}^n \frac{t_i-t_{i-1}}{\delta_G(\gamma_s(t_i))} + \frac{\epsilon}{2},
		\]
		which establishes the condition \eqref{first condition}.
		
		For the  condition \eqref{second condition},  since \(\gamma_s\) is an arc-length parametrization, we have
		\[
		d(\gamma_s(s), \gamma_s(t)) \leq |s - t| \quad \text{for all } s, t \in [0, L].
		\]
		Therefore, if we require that the subinterval length satisfies
		\(t_i - t_{i-1} < \epsilon M_1\), then for any \(t \in [t_{i-1}, t_i]\),
		\[
		d(\gamma_s(t), \gamma_s(t_i)) \leq t_i - t < \epsilon M_1 \leq \epsilon \,\delta_G(\gamma_s(t_i)).
		\]
    This means \(\gamma_i \subset B(x_i, \epsilon\, \delta_G(\gamma_s(t_i)))\), and the  condition \eqref{second condition} holds.
		
		To simultaneously satisfy both conditions, choose
		\[
		M = \min\left\{ M_2, \epsilon M_1 \right\}.
		\]
		By taking a partition \(P: 0 = t_0 < t_1 < \cdots < t_n = L\) with mesh size \(\|P\| < M\), the conditions \eqref{first condition} and \eqref{second condition} hold,  which completes the proof of this claim.

\medskip
	We now proceed with the proof. Let \(P: 0 = t_0 < t_1 < \cdots < t_n = L\) be the partition in Claim \ref{section3-claim}. 
	For each \( 1 \leq i \leq n \), we choose successive points \( x_{i-1} = x_{i,0}, \, x_{i,1}, \cdots, \, x_{i,n_i} = x_i \) of \( \gamma_i \) such that
\begin{equation}\label{subcurve length bounded}
	\ell(\gamma_i) \leq \sum_{j=1}^{n_i} d(x_{i,j-1}, x_{i,j}) + \epsilon / n.
\end{equation}
The condition \eqref{second condition} and the estimate \eqref{QH prop 3} yield that, for all \( 1 \leq i \leq n \) and \( 1 \leq j \leq n_i \),
\begin{equation}\label{subcurve lk bounded}
	\frac{d(x_{i,j-1}, x_{i,j})}{\delta_G(x_i)} \leq (1+2\epsilon) k_G(x_{i,j-1}, x_{i,j}).
\end{equation}
It follows from \eqref{first condition}, \eqref{subcurve length bounded}, \eqref{subcurve lk bounded} and the fact \( t_i - t_{i-1} = \ell(\gamma_i) \) that
	\[
	\ell_{\text{qh}}(\gamma) \leq \sum_{i=1}^n \frac{t_i-t_{i-1}}{\delta_G(\gamma_s(t_i))} +\epsilon\leq \sum_{i=1}^n\sum_{j=1}^{n_i} \frac{d(x_{i,j-1}, x_{i,j})}{\delta_G(x_i)} + 2\epsilon\leq (1+2\epsilon) \sum_{i=1}^n\sum_{j=1}^{n_i} k_G(x_{i,j-1}, x_{i,j}) + 2\epsilon.
	\]
	Notice that the double sum is at most \( \ell_{k}(\gamma) \). Since \( \epsilon \) is arbitrary, this yields the desired inequality.
	
It remains to show that the metric space $(G, k_G)$ is a length space. Let $x, y \in G$ and $\epsilon>0$. Since $G$ is rectifiably connected, by the definition of $k_G$ there exists a curve $\gamma$ joining $x$ and $y$ such that $\ell_{\text{qh}}(\gamma) < k_G(x,y)+\epsilon$. Using $\ell_k(\gamma)=\ell_{\text{qh}}(\gamma)$, we obtain $\ell_{k}(\gamma) < k_G(x,y)+\epsilon$. The arbitrariness of $\epsilon$ implies that $(G, k_G)$ is a length space.
\end{proof}


\section{Proofs of Theorems~\ref{diam John is John} and \ref{distance uniform is uniform}}
Before proceeding to the proofs of the main theorems, we first state and prove several auxiliary lemmas. We note from Proposition~\ref{t1} that, for every rectifiable curve $\gamma$ in a rectifiably connected domain, we always have $\ell_k(\gamma)=\ell_{\text{qh}}(\gamma)$.  So, for simplicity, throughout this section we use $\ell_k(\gamma)$ to denote the quasihyperbolic length of $\gamma$ in a rectifiably connected domain.

\begin{lem}\label{doubling property}$($\rm \cite[Lemma 4.1.11]{HKST15}$)$
	Let \(X \) be a \( Q \)-doubling metric space and \(G \subsetneq X\) be a domain. Fix \( R > 0 \) and \( a \geq 1 \), and let \( r = \frac{R}{a} \). Then for any \( x \in G\), the ball \( {B}(x, R) \) contains at most \( Q^{\lfloor\log_2 a\rfloor} \) disjoint balls of radius \( r \). Here and hereafter, \( \lfloor \cdot \rfloor \) means the greatest integer part.
\end{lem}

\begin{lem}\rm\label{ball sequence}
	Let $\mu\in (0,1)$ and \( x, y \in G \). If $d(x,y)> \mu \max\{\delta_G(x),\delta_G(y)\}$, and \( \gamma \subset G \) is a rectifiable curve joining $x$ and $y$, then there exists a sequence of points $x=z_1,z_2,\dots,z_{k+1}=y$ on $\gamma$ and a finite sequence of balls \(\{\overline{B}_i\}_{i=1}^k\) in \(G\) such that
	\begin{itemize}
		\item[(1)] For each \( i \in \{1, \ldots, k\} \), \( \overline{B}_i =\overline{B}(z_i, r_i) \) with \( r_i = \mu\,\delta_G(z_i) \), where \( z_i \in \gamma[z_{i-1}, y] \), but \( z_i \notin B_{i-1} \);
		\item[(2)] \( z_{k+1} =y \in \overline{B}_{k} \);
		\item[(3)] For any pair \(\{i, j\} \subset \{1, \ldots, k\}\) with \(|j - i| > 1\), \( \overline{B}_i \cap \overline{B}_j = \emptyset \) provided that \( k \geq 3 \);
		\item[(4)] For each \( i \in \{1, \ldots, k-1\} \),  \( \overline{B}_i \cap \overline{B}_{i+1} \neq \emptyset \).
	\end{itemize}
\end{lem}
\begin{proof}
The argument proceeds in two steps: 
	First, we shall prove that there exists a finite sequence $x=x_1,x_2,\ldots,x_{k_0},x_{k_0+1}=y$ on $\gamma$ such that
	\begin{equation}\label{eq-x-i-1}
	d(x_i,x_{i+1})=\mu\,\delta_G(x_i),\,\, 1\leq i\leq k_0-1
	\end{equation}
	and 
	 $$d(x_{k_0},y)\leq\mu\,\delta_G(x_{k_0}).
	$$
	To this end, set \( x_1 = x \). Since $d(x,y)> \mu \max\{\delta_G(x),\delta_G(y)\}$,  choose $x_2$ to be the last point on $\gamma$ along the direction from $x$ to $y$ such that \( x_2\in\gamma \cap S(x, \mu\,\delta_G(x))\).
	
	If \( y \in \overline{B}(x_2, \mu\,\delta_G(x_2)) \), then we take \( k_0=2 \). If \( y \notin \overline{B}(x_2, \mu\,\delta_G(x_2)) \), choose $x_3$ to be the last point on $\gamma[x_2,y]$ along the direction from $x_2$ to $y$ such that \( x_3\in\gamma \cap S(x_2, \mu\,\delta_G(x_2))\).

	If \( y \in \overline{B}(x_3, \mu\,\delta_G(x_3)) \), then we take \( k_0=3 \). If \( y \notin \overline{B}(x_3, \mu\,\delta_G(x_3)) \),
repeat this procedure until there exist $\nu\geq 3$ and a point \( x_\nu \in \gamma[x_{\nu-1}, y] \) such that \( y \in \overline{B}(x_\nu, \mu\,\delta_G(x_\nu)) \). We take \( k_0=\nu \).
	
Second,	select a subsequence \(\{z_j\}_{j=1}^{k}\) of \(\{x_i\}_{i=1}^{k_0}\) and get the ball sequence \(\{\overline{B}_j\}_{j=1}^{k}\), where \( \overline{B}_j=\overline{B}(z_j, \mu\,\delta_G(z_j)) \).
	
	Let \(\{x_i\}_{i=1}^{k_0+1}\) be the sequence constructed as above. For each \(i \in \{1, \ldots, k_0\}\), set
	\[
	\overline{B}_{1,i} = \overline{B}(x_i, \mu\,\delta_G(x_i)).
	\]
Then \eqref{eq-x-i-1} implies that $\overline{B}_{1,i}\cap \overline{B}_{1,i+1}\not=\emptyset$ for all $1\leq i<k_0-1$. We are going to select the desired sequence of balls from \(\{\overline{B}_{1,i}\}_{i=1}^{k_0}\).
	Set \( \overline{B}_1 = \overline{B}_{1,1} \) and \( z_1 = x_1 \). We define
		\[
		r_1 = \max\bigl\{r: r \in \{2, \ldots, k_0\} \text{ and } \overline{B}_1 \cap \overline{B}_{1,r} \neq \emptyset\bigr\}.
		\]
		
	Set \( \overline{B}_2 = \overline{B}_{1,r_1} \) and \( z_2 = x_{r_1} \).
	If \( r_1 = k_0 \), then we find the needed sequence of balls \(\{\overline{B}_i\}_{i=1}^{k}\) with \( k=2 \).
	If \( r_1 < k_0 \), then define
			\[
			r_2 = \max\bigl\{r: r \in \{r_1 + 1, \ldots, k_0\} \text{ and } \overline{B}_2 \cap \overline{B}_{1,r} \neq \emptyset\bigr\}.
			\]
	
Set \( \overline{B}_3 =\overline{B}_{1,r_2} \) and \( z_3 = x_{r_2} \).
If \( r_2 = k_0 \), then we find the needed sequence of balls \(\{\overline{B}_i\}_{i=1}^{k}\) with \( k=3 \).
		 If \( r_2 < k_0 \), then define
			\[
			r_3 = \max\bigl\{r: r \in \{r_2 + 1, \ldots, k_0\} \text{ and } \overline{B}_3 \cap \overline{B}_{1,r} \neq \emptyset\bigr\}.
			\]
	
By repeating this procedure, we find an integer \( k \leq k_0 \) such that
		\[
		\max\bigl\{r : r \in \{r_{k-1} + 1, \ldots, k_0\} \text{ and } \overline{B}_{k-1} \cap \overline{B}_{1,r} \neq \emptyset\bigr\} = k_0.
		\]
		Then set \(\overline{B}_{k} = \overline{B}_{1,k_0} \), \( z_{k} = x_{k_0} \) and \( z_{k+1} = y \). 

	In this way, we find the desired subsequence \(\{z_j\}_{j=1}^{k}\) of \(\{x_i\}_{i=1}^{k_0}\) and the corresponding ball sequence \(\{\overline{B}_j\}_{j=1}^{k}\). By construction, $\{z_j\}^{k}_{j=1}$ and $\{\overline{B}_j\}^{k}_{j=1}$ satisfy properties (1)–(4) in the lemma.
\end{proof}

\begin{lem}\rm\label{two-ball-int}
	Suppose that $X$ is a metric space and $G\subsetneq X$ is a locally $(q,\lambda)$-quasiconvex domain. Let $x, y\in G$. By Proposition~\ref{t1}, there exists a rectifiable curve $\gamma$ in $G$ joining $x$ and $y$ such that 	$\ell_k(\gamma) \leq  2 k_G(x,y)$. Then the following statements hold.
	\begin{itemize}
		\item[{\bf (i)}] If $\overline{B}(x, \frac{q}{8\lambda}\,\delta_G(x)) \cap \overline{B}(y, \frac{q}{8\lambda}\,\delta_G(y)) \neq \emptyset$, then the curve $\gamma$ satisfies
		\begin{equation*}
			\ell(\gamma) \le e^{q}\min\{\delta_G(x),\delta_G(y)\}
		\end{equation*}
		and
		\begin{equation*}
		e^{-q}\max\{\delta_G(x),\delta_G(y)\}  \le \delta_G(w) \le e^{ q}\min\{\delta_G(x),\delta_G(y)\} \quad \text{for all } w\in \gamma.
		\end{equation*}
		\item[{\bf (ii)}] If $d(x,y)\leq \frac{q}{8\lambda}\max\{\delta_G(x),\delta_G(y)\}$, then the curve $\gamma$ is a length $c^\prime$-uniform curve, where
		\[
		c^\prime = \frac{8\lambda e^{2q}}{\log\left(1+\frac{q}{8\lambda}\right)}.
		\]
	\end{itemize} 
\end{lem}

\begin{proof}
	{\bf (i)} Since $ \overline{B}(x, \frac{q}{8\lambda}\,\delta_G(x)) \cap \overline{B}(y, \frac{q}{8\lambda}\,\delta_G(y)) \neq \emptyset$, there exists a point  \( z_{0} \in \overline{B}(x, \frac{q}{8\lambda}\,\delta_G(x)) \cap \overline{B}(y, \frac{q}{8\lambda}\,\delta_G(y)) \). Then, by \eqref{QH prop 4}, we have 
	\begin{equation*}\label{x,z_0 k upper bounded}
		k_G(x,z_0)\leq 2\lambda\frac{d(x,z_0)}{\delta_G(x)}\leq \frac{q}{4}
	\end{equation*}
	and
	\begin{equation*}\label{z_0,y k upper bonded}
		k_G(z_0,y)\leq 2\lambda\frac{d(z_0,y)}{\delta_G(y)}\leq \frac{q}{4}.
	\end{equation*}
	Therefore,
	\begin{equation}\label{ell-k}
		\ell_k(\gamma) \leq  2 k_G(x,y)\leq 2(k_G(x,z_0)+k_G(z_0,y))\leq q.
	\end{equation}
	Then it follows from \eqref{QH prop 1} and \eqref{ell-k} that for all $w\in \gamma$, 
	\begin{equation}\label{eq-x-to-y}
		\left| \log\frac{\delta_G(x)}{\delta_G(w)}\right| \leq\ell_k(\gamma[x, w])\leq \ell_k(\gamma)\leq q.
	\end{equation}
	This implies that for all $w\in \gamma$, 
	\begin{equation*}
		e^{-q}\,\delta_G(x) \le \delta_G(w) \le e^{q}\,\delta_G(x).
	\end{equation*}
	By the same reasoning, \eqref{eq-x-to-y} holds with $x$ and $\gamma[x, w]$ replaced by $y$ and $\gamma[w, y]$, respectively. Therefore, we obtain \begin{equation*}
		e^{-q}\,\delta_G(y) \le \delta_G(w) \le e^{q}\,\delta_G(y).
	\end{equation*}
	
	Moreover, applying \eqref{QH prop 1} and \eqref{ell-k} again, we have
	\begin{equation*}
		\log\left( 1+\frac{\ell(\gamma)}{\min\{\delta_G(x),\delta_G(y)\}}\right)  \le \ell_k(\gamma) \leq q,
	\end{equation*}
	which  gives 
	$\ell(\gamma)\leq e^{q}\min\{\delta_G(x),\delta_G(y)\}.$ Hence the proof of statement {\bf (i)} is complete.
	\medskip
	
	{\bf (ii)} If $d(x,y)\leq \frac{q}{8\lambda}\max\{\delta_G(x),\delta_G(y)\}$,  then $\overline{B}(x, \frac{q}{8\lambda}\,\delta_G(x)) \cap \overline{B}(y, \frac{q}{8\lambda}\,\delta_G(y)) \neq \emptyset$. 
	Without loss of generality, we may assume that $\delta_G(x)\geq\delta_G(y)$. Then $y\in \overline{B}(x, \frac{q}{8\lambda}\,\delta_G(x))$.
	 Combining the fact that $\ell_k(\gamma) \leq  2 k_G(x,y)$ with  \eqref{QH prop 4} yields 
	\begin{equation}\label{beta lk qc}
		\ell_k(\gamma)\leq 2 k_G(x,y)\leq 4 \lambda\frac{d(x,y)}{\delta_G(x)}\leq \frac{ q}{2}.
	\end{equation}
	
	To prove that $\gamma$ is quasiconvex, we consider the following two cases.
	If \(\gamma \subset \overline{B}(x,\frac{q}{8\lambda}\,\delta_G(x)) \), then for each $w\in\gamma$, we have
	\begin{equation}\label{lem4.4 delta w delta x}
		\delta_G(w) \leq \delta_G(x)+d(x,w)\leq\left( 1+\frac{q}{8\lambda}\right) \delta_G(x).
	\end{equation}
	Then it follows from \eqref{beta lk qc} and \eqref{lem4.4 delta w delta x} that 
	\begin{equation}\label{lem4.4 qc-1}
		\ell(\gamma) \leq \left( 1+\frac{q}{8\lambda}\right) \delta_G(x) \ell_k(\gamma) \leq  \left( 4\lambda+\frac{q}{2}\right) d(x,y).
	\end{equation}

	If \(\gamma \not\subset \overline{B}( x,\frac{q}{8\lambda}\,\delta_G(x)) \), then we can choose a point \(z_0 \in \gamma \setminus \overline{B}( x,\frac{q}{8\lambda}\,\delta_G(x)) \).
	Applying \eqref{QH prop 1} and \eqref{beta lk qc}, we obtain
	
	\[
	\log\left( 1+\frac{q}{8\lambda}\right)  \leq \log\left( 1+\frac{d(x,z_0)}{\delta_G(x)}\right)\leq  \ell_k(\gamma) \leq  4\lambda\frac{d(x,y)}{\delta_G(x)}.
	\]
	Therefore, $$\delta_G(x)\leq\frac{4 \lambda}{\log\left( 1+\frac{q}{8\lambda}\right) }d(x,y).$$
	Combining this with \eqref{QH prop 1} and \eqref{beta lk qc} yields
	\begin{equation}\label{lem4.4-qc-2}
		\ell(\gamma)\leq e^{\ell_k(\gamma)}\,\delta_G(x)\leq  \frac{4 \lambda e^{ q/2} }{\log\left( 1+\frac{q}{8\lambda}\right) }d(x,y)\leq \frac{8\lambda e^{2q} }{\log\left( 1+\frac{q}{8\lambda}\right) }d(x,y) =: c^\prime d(x, y).
	\end{equation}
	Clearly, $c^\prime\geq8\lambda\geq 4\lambda+q/2$.
	Thus, \eqref{lem4.4 qc-1} and \eqref{lem4.4-qc-2} show that $\gamma$ is $c^\prime$-quasiconvex. 
	
	Moreover, by the statement {\bf (i)},  for all $w\in\gamma$ we have
	$$\ell(\gamma)\leq e^{q}\,\delta_G(x)\leq e^{ 2q}\,\delta_G(w),$$
	which implies that $\gamma$ is a length $e^{2 q}$-John curve. 
	
	Recalling that $\gamma$ is a $c'$-quasiconvex curve and $e^{2q}\leq c^\prime$, it follows that $\gamma$ is a length $c^\prime$-uniform curve.
\end{proof}

\begin{lem}\rm\label{pre-thm}
	Let $Q>1$ and let $X$ be a $Q$-doubling metric space. Suppose that \( G\subsetneq X \)  is a locally $(q,\lambda)$-quasiconvex domain. Let $x, y\in G$ be points satisfying $\delta_G(y)\leq 2\,\delta_G(x)$ and $d(x,y)>\frac{q}{8\lambda}\max\{\delta_G(x),\delta_G(y)\}$, and let $c\geq 1$. If $x$ and $y$ can be joined by a rectifiable curve $\gamma$ such that  
	$$	\diam(\gamma[x, w])\leq c\, \delta_G(w) \quad \text{for all } w\in \gamma,$$
	then there exists a rectifiable curve $\beta$ joining $x$ and $y$ such that  $\delta_G(x)\leq (1+c)e^q\,\delta_G(w)$ for all $w\in \beta$, and
	\begin{equation*}
		\ell(\beta)\leq 6ce^{q} Q^{\,\lfloor\log_2 (24\lambda c(1+c)/q)\rfloor}\, \delta_G(x).
	\end{equation*} 
\end{lem}
\begin{proof}
	Applying Lemma~\ref{ball sequence} with $\mu=\frac{q}{8\lambda}$, one can find a sequence of points $\{z_i\}_{i=1}^k$ on $\gamma$ and a sequence of balls $\{\overline{B}_i=\overline{B}(z_i,r_i)\}_{i=1}^{k}$ with $r_i=\frac{q}{8\lambda}\,\delta_G(z_i)$ satisfying the description given in Lemma~\ref{ball sequence}.
	 Since 	$\diam(\gamma[x, w])\leq c \,\delta_G(w)$ for all  $w\in \gamma$  and $\delta_G(y) \le 2\,\delta_G(x)$, we have
	\begin{equation}\label{x,zi upper bounded}
		d(x, z_i) \le \diam(\gamma[x,z_i]) \le \diam(\gamma)\leq c\,\delta_G(y)\leq 2c\,\delta_G(x)
	\end{equation}
	and hence
	$$
	\frac{8\lambda}{q}r_i=\delta_G(z_i) \le \delta_G(x) + d(x, z_i) \le (1+2c)\,\delta_G(x).
	$$
	Therefore, 
	$$
	d(x, z_i) + r_i \le\left(2c+\frac{q(1+2c)}{8\lambda}\right)\delta_G(x)< 3c\,\delta_G(x)
	$$
	which implies
	\begin{equation}\label{ball in big ball}
		\overline{B}(z_i, r_i) \subset {B}(x, 3c\,\delta_G(x)).
	\end{equation}
	
	Let \(R=3c\,\delta_G(x)\) and \(a=24\lambda c(1+c)/q\).
	Since $$ \delta_G(x) \leq \delta_G(z_i) + d(x, z_i) \le \delta_G(z_i) + \diam(\gamma[x, z_i]) \le (1+c)\,\delta_G(z_i), $$ it follows that
	\[
	\frac{R}{a}=\frac{3qc\,\delta_G(x)}{24\lambda c(1+c)} \le \frac{q}{8\lambda}\,\delta_G(z_i).
	\]
	
To deduce the upper bound on $k$, observe that, by property (3) of Lemma~\ref{ball sequence}, balls with
non-adjacent indices are pairwise disjoint. Hence the odd-indexed family ${\overline{B}_1, \overline{B}_3, \overline{B}_5, \ldots}$ and
the even-indexed family ${\overline{B}_2, \overline{B}_4, \overline{B}_6, \ldots}$ are each pairwise disjoint. Applying Lemma~\ref{doubling property} to each family separately, each contains at most $Q^{\lfloor\log_2 a\rfloor}$ balls, yielding
   	\begin{equation}\label{k upper bounded}
		k \leq 2Q^{\,\lfloor\log_2 (24\lambda c(1+c)/q)\rfloor}.
	\end{equation}
	
By Lemma~\ref{ball sequence}(2) and Lemma~\ref{ball sequence}(4), for each  $1\leq i\leq k$, we have $\overline{B}(z_i, \frac{q}{8\lambda}\,\delta_G(z_i)) \cap \overline{B}(z_{i+1}, \frac{q}{8\lambda}\,\delta_G(z_{i+1})) \neq \emptyset$. Applying  Lemma~\ref{two-ball-int}{\bf (i)} yields that for each  $1\leq i\leq k$, there exists a rectifiable curve $\alpha_i$ joining $z_i$ to $z_{i+1}$ in $G$ such that for any $w\in \alpha_i$,
\begin{equation}\label{w zi}
	e^{-q}\,\delta_G(z_i) \le \delta_G(w) \le e^{q}\,\delta_G(z_i)
\end{equation}
and
\begin{equation}\label{two ball John}
	\ell(\alpha_i)	\leq e^{q}\,\delta_G(z_{i}).
\end{equation}
	
	Let $\beta=\bigcup\limits_{i=1}^{k}\alpha_i$. 
	For each \( w \in \beta \), there exists \( t \in \{1, \dots, k\} \) such that \( w \in \alpha_t\).  Since $\diam(\gamma[x, z_t])\leq c\,\delta_G(z_t)$,  it follows from  \eqref{w zi} that
	$$\delta_G(x)\leq \delta_G(z_t)+\diam (\gamma[x, z_t])\leq (1+c)\, \delta_G(z_t) \leq (1+c) e^q\, \delta_G(w).$$
Moreover, by \eqref{x,zi upper bounded}, \eqref{k upper bounded} and \eqref{two ball John}, we obtain
\begin{align*}
			\ell(\beta) &\leq \sum_{i=1}^k \ell(\alpha_i) \leq e^{q}\sum_{i=1}^k \delta_G(z_{i}) \leq e^{q}\sum_{i=1}^k \left( \delta_G(x)+d(x, z_i)\right) \\
			&\leq e^{q}k(1+2c)\, \delta_G(x) \leq 3ce^{q}k \,\delta_G(x)\leq 6ce^{q} Q^{\,\lfloor\log_2 (24\lambda c(1+c)/q)\rfloor} \,\delta_G(x).
\end{align*}	
Hence the proof is complete.
\end{proof}

\begin{lem}\rm\label{diameter implies length}
Let $Q>1$ and let $X$ be a $Q$-doubling metric space. Suppose that \( G\subsetneq X \)  is a locally $(q,\lambda)$-quasiconvex domain. Then the following hold:
\begin{itemize}
	\item[(1)] If $G$ is a diameter John domain, then $G$ is length John;
	
	\item[(2)] If $G$ is a diameter uniform domain, then $G$ is length uniform.
	\end{itemize}
\end{lem}
\begin{proof} (1) Assume that $G$ is a diameter $c$-John domain with $c\geq 1$. To show that $G$ is length John, it suffices to show that there exists a constant $c'\geq 1$  such that  each pair of points $x, y\in G$ can be joined by a length $c'$-John curve.
 
Let $x,y\in G$. We consider the following cases.

{\bf Case 1.}  \( d(x,y) \leq \frac{q}{8\lambda} \max\{\delta_G(x),\delta_G(y) \}\).  By Lemma \ref{two-ball-int}{\bf (ii)}, there exists a rectifiable curve $\beta$ that is a length $c_1'$-uniform curve, and consequently a length $c_1'$-John curve, where
$$c_1^\prime=\frac{8 \lambda e^{2 q} }{\log\left( 1+\frac{q}{8\lambda}\right)}.$$

{\bf Case 2.}  \( d(x,y) > \frac{q}{8\lambda} \max\{\delta_G(x),\delta_G(y) \}\).  
Since $G$ is diameter $c$-John, there exists a curve $\zeta$ joining $x$ and $y$ such that 	for all \( w \in \zeta \),
	\begin{equation}\label{John-thm}
	\min\{\diam(\zeta[x,w]) , \diam(\zeta[w,y])\} \leq c \, \delta_G(w) 
	\end{equation} 
	Let \( w_0 \in \zeta \) be such that
	\[
	\diam(\zeta[x, w_0]) = \diam(\zeta[y, w_0]).
	\]
	Then for all  $w \in \zeta[x, w_0]$, it follows from \eqref{John-thm} that
	\begin{equation}\label{diam john}
	d(x, w) \leq \diam(\zeta[x,w]) \leq c\,\delta_G(w).
	\end{equation}	
Then we divide this case into the following two subcases.

	\underline{Subcase 2.1.} \(\delta_G(w_0) \le 2\,\delta_G(x) \). If $d(x,w_0)\leq \frac{q}{8\lambda}\max\{\delta_G(x),\delta_G(w_0)\}$, then $\overline{B}(x, \frac{q}{8\lambda}\,\delta_G(x)) \cap \overline{B}(w_0, \frac{q}{8\lambda}\,\delta_G(w_0)) \neq \emptyset$. Applying Lemma~\ref{two-ball-int}(i),  we obtain a rectifiable curve $\beta_1$ joining $x$ and $w_0$ such that 
	\begin{equation}\label{subcase2.1 beta1 x1}
		\ell(\beta_1)\leq e^{q}\min\{\delta_G(x),\delta_G(w_0)\}\leq e^q\delta_G(x)
	\end{equation}
	and
	\begin{equation}\label{subcase2.1 x w01}
		\delta_{G}(x)\leq e^q\delta_{G}(w)  
	\end{equation}
	for all $w\in\beta_1$.
	
	If $d(x,w_0)> \frac{q}{8\lambda}\max\{\delta_{G}(x),\delta_{G}(w_0)\}$, applying Lemma \ref{pre-thm} with $y=w_0$, then there is a rectifiable curve $\beta_1$ joining $x$ and $w_0$ such that  
	\begin{equation}\label{subcase2.1 x w02}
	\delta_G(x)\leq (1+c)e^q\,\delta_G(w)
	\end{equation}
	 for all $w\in \beta_1$, and
	\begin{equation}\label{ell-beta-1}
		\ell(\beta_1)\leq 6ce^{q} Q^{\,\lfloor\log_2 (24\lambda c(1+c)/q)\rfloor}\, \delta_G(x).
	\end{equation}
	 
	Thus, by combining \eqref{subcase2.1 beta1 x1}, \eqref{subcase2.1 x w01}, \eqref{subcase2.1 x w02} and \eqref{ell-beta-1}, we obtain a rectifiable curve $\beta_1$ joining $x$ and $w_0$ such that 
	\begin{equation}\label{ell-beta-2}
		\ell(\beta_1)\leq 6ce^{q} Q^{\,\lfloor\log_2 (24\lambda c(1+c)/q)\rfloor}\, \delta_G(x)
	\end{equation} 
	and for all $w\in \beta_1$,
	\begin{equation}\label{subcase2.1 x w03}
		\delta_G(x)\leq (1+c)e^q\,\delta_G(w).
	\end{equation}

	Consequently, for all $w\in \beta_1$,
\begin{equation}\label{case 11 john}
		\ell(\beta_1[x, w])\leq 	\ell(\beta_1)\leq 6ce^{q} Q^{\,\lfloor\log_2 (24\lambda c(1+c)/q)\rfloor} \,\delta_G(x)\leq  6c(1+c)e^{2q} Q^{\,\lfloor\log_2 (24\lambda c(1+c)/q)\rfloor}\, \delta_G(w).
\end{equation}
Moreover, since \( d(x,y) > \frac{q}{8\lambda} \max\{\delta_G(x),\delta_G(y) \}\), it follows from \eqref{ell-beta-2} that
\begin{equation}\label{case 11 qc}
		\ell(\beta_1)\leq \frac{48c\lambda e^{q}}{q} Q^{\,\lfloor\log_2 (24\lambda c(1+c)/q)\rfloor}\, d(x, y).
\end{equation}

For the same reason, there is a curve $\beta_2$ joining $y$ and $w_0$ such that for all $w\in \beta_2$,
\begin{equation}\label{case 12 join}
\ell(\beta_2[y,w])\leq  6c(1+c)e^{2q} Q^{\,\lfloor\log_2 (24\lambda c(1+c)/q)\rfloor} \,\delta_G(w).
\end{equation}
and 
\begin{equation}\label{case 12 qc}
\ell(\beta_2)\leq \frac{48c\lambda e^{q}}{q} Q^{\,\lfloor\log_2 (24\lambda c(1+c)/q)\rfloor} \,d(x, y).
\end{equation}

By \eqref{case 11 john}, \eqref{case 11 qc}, \eqref{case 12 join}  and \eqref{case 12 qc}, the curve $\beta=\beta_1\cup \beta_2$ is a length $c'_2$-uniform curve and hence a length $c'_2$-John curve, where
 $$c'_2=\frac{96c(1+c)\lambda e^{2q}}{q}  Q^{\,\lfloor\log_2 (24\lambda c(1+c)/q)\rfloor}.$$

\medskip
	\underline{Subcase 2.2.} \( \delta_G(w_0) > 2\,\delta_G(x) \).
	Then there exists an integer \(N\geq2\) such that
	\begin{equation}\label{w0 Nx}
		2^{N-1}\,\delta_G(x) <\delta_G(w_0) \le 2^{N}\,\delta_G(x).
	\end{equation}
Let $x_0=x$. For $1\leq j\leq N-1$, choose $x_j$ to be the first point on  $\zeta[x,w_0]$ along the direction from $x$ to $w_0$ such that 	$\delta_G(x_j)=2^{j}\,\delta_G(x).$ 
It follows from \eqref{w0 Nx} and the connectedness of $\zeta$ that such $x_j$ exists and $x_{N-1}\not=w_0$. By setting $x_N=w_0$,  we find a sequence of points $x=x_0, x_1, x_2,\dots,x_{N}=w_0$. Clearly, 
	\begin{equation}\label{xN-2}
	\delta_G(x_{j+1})= 2\,\delta_G(x_{j}),\quad 0\leq j\leq N-2
	\end{equation}
	and
	\begin{equation}\label{xN-1}
	\delta_G(x_{N-1})<\delta_G(x_{N})\leq 2\,\delta_G(x_{N-1}).
\end{equation}

For $0\leq j\leq N-1$,  it follows from \eqref{diam john} that for all $w\in \zeta[x_j, x_{j+1}]$, 
$$\diam (\zeta[x_j, w])\leq \diam(\zeta[x, w])\leq c \,\delta_G(w).$$
Then the same reasoning that leads to \eqref{ell-beta-2} and \eqref{subcase2.1 x w03} yields that there exists a rectifiable curve $\beta_j$ joining $x_j$ and $x_{j+1}$ such that  
\begin{equation}\label{delta xj w}
	\delta_G(x_j)\leq (1+c)e^q\,\delta_G(w)  \quad \text{for all } w\in \beta_j,
\end{equation} 
and
\begin{equation}\label{beta j delta xj}
	\ell(\beta_j)\leq 6ce^{q} Q^{\,\lfloor\log_2 (24\lambda c(1+c)/q)\rfloor} \,\delta_G(x_j).
\end{equation}

	Let $\beta_{1}^\prime=\bigcup_{j=0}^{N-1}\beta
	_j$. Set $b=6ce^{q} Q^{\,\lfloor\log_2 (24\lambda c(1+c)/q)\rfloor} $. For each $w \in \beta_{1}^\prime$, there exists $t\in \{0, \dots, N-1\}$ such that
	$w \in \beta_t$. Then, it follows from \eqref{xN-2}, \eqref{delta xj w} and \eqref{beta j delta xj}  that
	\begin{equation}\label{case 2 1 john}
		\ell\left(\beta_{1}^\prime[x, w]\right) \leq  \sum_{j=0}^{t}\ell(\beta_j ) \leq b \sum_{j=0}^{t} \delta_G(x_j)\leq 2b\,\delta_G(x_{t})\leq 2b(1+c)e^q\,\delta_G(w).
\end{equation}

For the same reason, there is a curve $\beta^\prime_{2}$ joining $y$ and $w_0$ such that for all $w\in\beta^\prime_{2}$, 
\begin{equation}\label{case 2 2 join}
\ell\left(\beta^\prime_{2}[y,w]\right)\leq  2b(1+c)e^q\,\delta_G(w).
\end{equation} 
Therefore, it follows from \eqref{case 2 1 john} and \eqref{case 2 2 join} that $\beta=\beta_{1}^\prime\cup\beta_{2}^\prime$ is a length $c'_3$-John curve, where
 $$c'_3=12c(1+c)e^{q} Q^{\,\lfloor\log_2 (24\lambda c(1+c)/q)\rfloor}.$$ 
 
 Combining these cases, with $c' = \max\{c'_1, c'_2, c'_3\}$,  we have shown that each pair of points $x, y\in G$ can be joined by a length $c'$-John curve, which completes the proof of $(1)$.
\medskip

 (2) Assume that $G$ is a diameter $c$-uniform domain with $c\ge 1$. To prove that it is a length uniform domain, it suffices to show that there exists a constant $c''\ge 1$ such that every pair of points $x,y\in G$ can be joined by a length $c''$-uniform curve. To this end, let $x,y\in G$ and follow the same proof of (1) by considering two cases: Case 1 and Case 2, where Case 2 consists of Subcase 2.1 and Subcase 2.2.
	
In Case 1 and Subcase 2.1, the rectifiable curve $\beta$ constructed in the corresponding case of (1) is already a length $c'$-uniform curve. Hence nothing remains to be proved in these two situations.

It now remains to consider Subcase 2.2. Since the rectifiable curve $\beta = \beta'_1 \cup \beta'_2$ in the corresponding case of (1) has been shown to be a length $c'$-John curve, it suffices to show that this curve is quasiconvex.

Since $G$ is a diameter $c$-uniform domain (hence also a diameter $c$-John domain), the curve $\zeta$ in \eqref{John-thm} additionally satisfies
$$\diam(\zeta)\leq c\, d(x, y).$$
Since we are in Case 2, we have $\delta_G(x)< \frac{8\lambda}{q}\, d(x, y)$. This, together with $\diam(\zeta)\leq c \,d(x, y)$, yields 
$$\delta_G(w_0)\leq \delta_G(x)+d(x,w_0)\leq \delta_G(x)+\diam(\zeta)\leq \left(  \frac{8\lambda}{q}+c\right) \, d(x,y).$$
Consequently, it follows from \eqref{xN-2}, \eqref{xN-1} and \eqref{beta j delta xj} that
\begin{equation*}\label{2 case 2 1 qc}
	\ell(\beta'_1)=\sum_{j=0}^{N-1} \ell(\beta_j) \leq b\sum_{j=0}^{N-1} \delta_G(x_j) \leq 2b\, \delta_G(x_{N-1})<2b\,\delta_G(w_0)\leq 2b\left(  \frac{8\lambda}{q}+c\right) \, d(x,y).
\end{equation*}
For the same reason, we obtain the same estimate for $\ell(\beta'_{2})$. Therefore,
$$\ell(\beta)=\ell(\beta'_1)+\ell(\beta'_2)\leq 4b\left( \frac{8\lambda}{q}+c\right)\, d(x,y)\leq  24ce^{q} Q^{\,\lfloor\log_2 (24\lambda c(1+c)/q)\rfloor}\left(  \frac{8\lambda}{q}+c\right) \, d(x,y),$$
which proves the quasiconvexity of $\beta$. 

Combining these cases, we have shown that each pair of points $x, y\in G$ can be joined by a length $c''$-uniform curve, where 
$$c''=\max\left\{c', 24ce^{q} Q^{\,\lfloor\log_2 (24\lambda c(1+c)/q)\rfloor}\Big( \frac{8\lambda}{q}+c\Big)\right\}.$$
This completes the proof of $(2)$.
\end{proof}

Adapting the argument from \cite[pp.~236-237]{Va04} (originally for Banach spaces) to the setting of metric spaces, we prove that a domain in a metric space is distance uniform if and only if it is diameter uniform. For the convenience of the reader, we provide a detailed proof of this result here.

\begin{lem}\rm\label{distance uniform is diameter uniform}
	Suppose that $X$ is a metric space. Then \( G\subsetneq X \) is a distance uniform domain if and only if it is diameter uniform.
\end{lem}

\begin{proof}
	Since a diameter uniform domain is by definition distance uniform, it suffices to show that a distance uniform domain is also diameter uniform.
	Let $G$ be a distance $c$-uniform domain for some $c\geq 1$. It suffices to show that there exists a constant $c^\prime\geq 1$ such that each pair of points $x, y\in G$ can be joined by a diameter $c^\prime$-uniform curve.
	
	Let \( x, y \in G \).  Then there exists a distance $c$-uniform curve $\gamma$ joining $x$ and $y$ satisfying 
	\begin{equation}\label{diam-turning}
		\diam(\gamma)\leq c\, d(x, y)
	\end{equation}
	and for all $w\in \gamma$,
	\begin{equation}\label{diam-John}
		\min\{d(x,w), d(w,y)\}\leq c\, \delta_G(w).
	\end{equation} 
	
	If $d(x,y) \leq \frac{1}{2} \max\{\delta_G(x), \delta_G(y)\},$
	without loss of generality, we may assume $\delta_G(x) \geq \delta_G(y)$.
	Then 
	\begin{equation}\label{delta-com}
		\delta_G(y) \geq \delta_G(x) - d(x,y) \geq \frac{1}{2}\, \delta_G(x).
	\end{equation}
	It follows from \eqref{diam-John} and \eqref{delta-com} that for all $w\in \gamma$, 
	$$\frac{1}{2}\,\delta_G(x)\leq \min\{\delta_G(x),\delta_G(y)\}\leq \delta_G(w)+\min\{d(x,w), d(w,y)\}\leq(1+c)\,\delta_G(w).$$
	This, together with \eqref{diam-turning}, yields that for all $w\in\gamma$,
	$$\diam(\gamma)\leq c\,d(x,y)\leq\frac{1}{2}c\,\delta_G(x)\leq c(1+c)\,\delta_G(w),$$
	which implies that $\gamma$ is a diameter  $c(1+c)$-John curve. This, together with \eqref{diam-turning}, yields that $\gamma$ is a  diameter $c(1+c)$-uniform curve joining $x$ and $y$.
	
	If $d(x,y)\geq \frac{1}{2}\max\{\delta_G(x),\delta_G(y)\}$,
	let \( z_0 \in \gamma \) be a point with \(d(x,z_0)=d(z_0,y)=:t \). Then, by \eqref{diam-turning} and \eqref{diam-John}, we have
	\begin{equation}\label{d t}
		\frac{1}{2}\,d(x,y)\leq\frac{1}{2}\,(d(x,z_0)+d(z_0,y))=t\leq \diam(\gamma)\leq c\,d(x,y)
	\end{equation}
	and
	\begin{equation}\label{t z0}
		t=d(x,z_0)=d(z_0,y)\leq c\,\delta_G(z_0).
	\end{equation}
	To find a diameter $c^\prime$-uniform curve  joining $x$ and $y$, by symmetry, it suffices 
	to find a rectifiable curve \(\beta\) joining $x$ and $z_0$ such that
	\begin{equation}\label{diam-turning-to}
		\diam(\beta) \leq c_1\,d(x,y)
	\end{equation}
	and for all $w\in \beta$,
	\begin{equation}\label{diam-John-to}
		\diam(\beta[x, w]) \leq c_1\,\delta_G(w),
	\end{equation}
	where $c_1\geq 1$ is a constant depending only on $c$.

	Setting \( r=\delta_G(x)/16 \), by \eqref{d t}, we have \( r\leq d(x,y)/8 \leq t/4 \). Hence there is a unique integer \( m \geq 3 \) such that
	
	\begin{equation}\label{t-r}
		2^{m-1}r \leq t < 2^m r.
	\end{equation}

	We define inductively points \( x_j\in G \) and curves \( \alpha_j, \, 1 \leq j \leq m-1 \), as follows. Let $x_0=x$ and $\alpha_0=\gamma[x,z_0]$. Choose $x_1$ to be the first point on $\alpha_0$ along the direction from $x$ to $z_0$ such that $x_1\in \alpha_0\cap S(x_0,r)$. Since $G$ is a distance $c$-uniform domain, there is a distance $c$-uniform curve $\alpha_1$ joining $x_1$ and $z_0$. Choose $x_2$ to be the first point on $\alpha_1$ along the direction from $x_1$ to $z_0$ such that $x_2\in\alpha_1\cap S(x,2r)$. Let $\alpha_2$ be a distance $c$-uniform  curve joining $x_2$ and $z_0$. Choose $x_3$ to be the first point on $\alpha_2$ along the direction from $x_2$ to $z_0$ such that $x_3\in\alpha_2\cap S(x,4r)$. By repeating this procedure, one obtains a sequence of points $x=x_0, x_1, \dots, x_{m-1}$ and a family of curves $\alpha_0, \alpha_1, \dots, \alpha_{m-1}$.
	
	From the construction, for $1\leq j \leq m-1$, we have $x_j\in \alpha_{j-1}\cap S(x, 2^{j-1}r)$ and $\alpha_j$ is a distance $c$-uniform curve joining $x_j$ and $z_0$. Therefore, for any $1\leq j \leq m-1$, it holds that
	\begin{equation}\label{d-1}
		d(x_j, x_{j-1})\geq d(x_j, x)-d(x, x_{j-1})=2^{j-1}r-2^{j-2}r= 2^{j-2}r
	\end{equation}
	and
	\begin{equation}\label{d-2}
		d(x_j, z_0)\geq d(z_0, x)-d(x, x_j) =t-2^{j-1}r\geq 2^{m-1}r-2^{j-1}r \geq 2^{j-2}r.
	\end{equation}
	
	Define successive curves \( \beta_0, \ldots, \beta_{m-1} \) by \(\beta_0 = \gamma[x_0, x_1], \, \beta_{m-1} = \alpha_{m-1} \) and by \( \beta_j=\alpha_j[x_j, x_{j+1}] \) for \( 1 \leq j \leq m-2 \).
	Let $\beta=\bigcup\limits_{j=1}^{m-1}\beta_j$.
	For each $0\leq j\leq m-2$, since $\bigcup\limits_{i=0}^{j}\beta_i\subset \overline{B}(x_0, 2^j r)$, we obtain
	\begin{equation}\label{diam  1,m-2}
		\diam(\beta_j)\leq \diam\Big(\bigcup_{i=0}^{j}\beta_i\Big)\leq  2^{j+1}r.
	\end{equation}
	Recalling that $\beta_{m-1}=\alpha_{m-1}$ is a distance $c$-uniform curve, by the fact that both $x_{m-1}$ and $z_0$ belong to $B(x, 2^m r)$, we obtain
	\begin{equation}\label{diam m-1}
		\diam(\beta_{m-1})\leq cd(x_{m-1},z_0)\leq 2^{m+1}cr.
	\end{equation}
	Moreover, by \eqref{d t} and \eqref{t-r}, we have \( 2^{m-1}r \leq t\leq c\,d(x,y) \).
	Therefore, it follows from \eqref{diam  1,m-2} and \eqref{diam m-1} that 
	\begin{equation}\label{diam-beta}
		\diam(\beta)\leq   2^{m-1}r+2^{m+1}cr\leq 2^{m+2}cr\leq  8c^2 \,d(x,y),
	\end{equation}
	which gives \eqref{diam-turning-to} with $c_1=8c^2.$
	
	We now prove that $\beta$ satisfies \eqref{diam-John-to}. Let \( w \in \beta_j \).  If \( j = 0 \), then since $\beta_0=\gamma[x_0, x_1]$ and  $d(x,w)\leq r\leq d(w,y)$, it follows from \eqref{diam-John} that
	$$d(x, w)=\min\{d(x,w), d(w,y)\}\leq c\,\delta_G(w).$$
	This, together with \eqref{diam  1,m-2}, yields 
	\begin{equation}\label{diam-John-0}
		\diam(\beta[x,w])\leq \diam (\beta_0)\leq 2r=\frac{1}{8}\,\delta_G(x)\leq\frac{1}{8}\,(\delta_G(w)+d(x,w))\leq\frac{1}{8}(1+c)\,\delta_G(w).
	\end{equation}

	If \( 1 \leq j \leq m - 2 \), then we  divide the argument into the following  two cases.
	
	\textbf{Case 1.} \(d(x_j,w)\leq 2^{j-3} r / c\). Since \( x_j \in \alpha_{j-1} \), and \( \alpha_{j-1} \) is a distance $c$-uniform curve, it follows from \eqref{d-1} and \eqref{d-2} that
	\[
	c\,\delta_G(x_j) \geq \min\{d(x_j,x_{j-1}),d(x_j,z_0)\}\geq 2^{j-2} r.
	\]
	Thus,
	$$\delta_G(w)\geq \delta_G(x_j)-d(x_j,w)\geq 2^{j-2}r/c-2^{j-3}r/c\geq2^{j-3}r/c.$$
	Combining this with \eqref{diam  1,m-2}, we obtain
	$$\diam(\beta[x,w])\leq 2^{j+1}r\leq 16c\,\delta_G(w).$$
	
	\textbf{Case 2.} \(d(x_j,w) \geq 2^{j-3} r / c\). Since $d(w, x)\leq d(x_{j+1}, x)=2^{j}r$, we have 
	$$d(w,z_0)\geq d(z_0, x)-d(x, w)\geq d(z_0, x)-d(x_{j+1}, x)=t-2^j r \geq 2^{m-2} r.$$ 
	Then the distance $c$-uniform condition for \( \alpha_j \) gives
	$$c\,\delta_G(w)\geq \min\{d(x_j,w), d(w,z_0)\} \geq  2^{j-3} r / c.$$
	This, along with \eqref{diam  1,m-2}, yields
	$$\diam(\beta[x,w])\leq  2^{j+1}r\leq16c^2\,\delta_G(w).$$
	
	Combining the two cases, we obtain
	\begin{equation}\label{diam-John-1}
		\diam(\beta[x,w])\leq  16c^2\,\delta_G(w).
	\end{equation}
	
	Finally assume that \( j = m - 1 \). Recall from \eqref{diam-beta} that \( \diam(\beta) \leq 2^{m+2} cr \). We consider the following three cases.
	
	\textbf{Case A.} \(d(x_{m-1},w)\leq 2^{m-4} r / c\).  Recall that the estimates in \eqref{d-1} and \eqref{d-2}  yield 
	$$d(x_{m-2},x_{m-1})\geq 2^{m-3} r \quad\text{and}\quad d(x_{m-1},z_0)\geq 2^{m-3} r.$$
	Since $x_{m-1}\in \alpha_{m-2}$ and  $\alpha_{m-2}$ is a distance $c$-uniform curve,  we have 
	$$c\,\delta_G(x_{m-1}) \geq \min\{d(x_{m-2},x_{m-1}), d(x_{m-1},z_0)\}\geq 2^{m-3}r. $$
	Thus,
	$$\delta_G(w)\geq\delta_G(x_{m-1})-d(x_{m-1},w)\geq 2^{m-3}r/c-2^{m-4}r/c\geq 2^{m-4}r/c.$$
	Combining this with \( \diam(\beta) \leq 2^{m+2} cr \), we obtain
	$$\diam(\beta[x,w])\leq \diam(\beta)\leq 2^{m+2}cr\leq 64c^2\,\delta_G(w).$$
	
	\textbf{Case B.} \(d(w,z_0) \leq 2^{m-4} r / c\). By \eqref{t z0} and \eqref{t-r}, we get
	$$\delta_G(z_0) \geq t/c \geq2^{m-1} r / c $$
	which implies
	$$ \delta_G(w) \geq \delta_G(z_0)-d(z_0,w)\geq 2^{m-1}r/c-2^{m-4}r/c\geq2^{m-2}r/c.$$
	Consequently,
	$$\diam(\beta[x,w])\leq\diam(\beta)\leq 2^{m+2}cr\leq 16c^2\,\delta_G(w).$$
	
	\textbf{Case C.} \(\min\{d(x_{m-1},w), d(w,z_0)\}> 2^{m-4} r / c\). Since  \( \alpha_{m-1} \) is a  distance uniform curve, we have 
	$$\delta_G(w) \geq \min\{d(x_{m-1},w), d(w,z_0)\} /c>2^{m-4} r / c^2.$$
	Therefore,
	$$\diam(\beta[x,w])\leq \diam(\beta)\leq 2^{m+2}cr\leq64 c^3\,\delta_G(w).$$

	The three cases together yield
	\begin{equation}\label{diam-John-2}
		\diam(\beta[x,w])\leq  64c^3\,\delta_G(w).
	\end{equation}

Combining  \eqref{diam-John-0}, \eqref{diam-John-1} and \eqref{diam-John-2}, we conclude that \eqref{diam-John-to} holds with $c_1=64 c^3$. This completes the proof.
\end{proof}

\begin{proof}[\textbf{Proof of Theorems \ref{diam John is John} and \ref{distance uniform is uniform}}]
Because a length John domain is diameter John, we only need to prove that a diameter John domain is length John. The proof of Theorem \ref{diam John is John} follows from part (1) of Lemma \ref{diameter implies length}.

For Theorem \ref{distance uniform is uniform}, the implications $(i)\Rightarrow (ii)$ and  $(ii)\Rightarrow (iii)$ are immediate  from the definition. Moreover, $(iii)\Rightarrow (i)$  follows by combining Lemma \ref{distance uniform is diameter uniform} with part (2) of Lemma \ref{diameter implies length}.
\end{proof}

 \begin{proof}[\textbf{Proof of Corollary \ref{John-cor}}]
By Theorems \ref{diam John is John} and \ref{distance uniform is uniform},   to prove Corollary \ref{John-cor}, it is enough to show that $G\subsetneq X$ is locally quasiconvex. In fact, since $X$ is a length space,   it is  $c$-quasiconvex  for any $c>1$. Hence, by Lemma \ref{qc is locally qc}, $G$ is locally quasiconvex. This completes the proof.
	\end{proof}

\section*{Acknowledgments}
	The first author (Yaxiang Li) was partly supported by NSFC under Grant No. 12371071 and  NSF of Hunan Province under Grant No. 2026JJ50357. The second author (Yahui Sheng) was partly  supported by NSFC under Grant No. 12571081. The third author (Zhuang Wang) was partly supported by NSF of Hunan Province under Grant No. 2024JJ6299 and the Scientific Research Project of the Hunan Provincial Education Department  under Grant No. 25B0095.
	
	We sincerely thank Prof. Manzi Huang for her careful reading of the manuscript and her valuable comments and suggestions, which have greatly improved this work.

\bigskip

\noindent Yaxiang  Li\\  E-mail address: yaxiangli@163.com\medskip

\noindent School of Mathematics and Statistics, Hunan First Normal University, Changsha, Hunan 410205, P. R. China.

\bigskip
\bigskip

\noindent Yahui Sheng\\E-mail address: yahuisheng\_26@163.com\medskip

\noindent Zhuang Wang\\ E-mail address: zwang@hunnu.edu.cn\medskip

\noindent
Key Laboratory of Computing and Stochastic Mathematics (Ministry of Education), School of Mathematics and Statistics, Hunan Normal University, Changsha, Hunan 410081, P. R. China.

\end{document}